\documentclass{amsart}

\usepackage{graphicx}
\usepackage{comment}
\usepackage{mathrsfs}
\usepackage{amsmath}
\usepackage{amssymb}
\usepackage{amsthm}
\usepackage{amsfonts}

\newtheorem{theorem}{Theorem}[section]
\newtheorem{lemma}[theorem]{Lemma}
\newtheorem{prop}[theorem]{Proposition}

\theoremstyle{definition}

\newtheorem{example}[theorem]{Example}

\newtheorem{remark}{Remark}

\numberwithin{equation}{section}

\begin{document}

\title[Teichm\"uller modular group and generalized Cantor set]{On countability of Teichm\"uller modular groups for analytically infinite Riemann surfaces defined by generalized Cantor sets}

\author{Erina Kinjo}
\address{Department of Mechanical Engineering,  
Ehime University, 3 Bunkyo-cho, Matsuyama, Ehime 790-8577, Japan}
\email{kinjo.erina.ax@ehime-u.ac.jp}


\date{July 6, 2024}



\begin{abstract}
For any analytically finite Riemann surface, the Teichm\"uller modular group is countable, but it is not easy to find an analytically infinite Riemann surface for which the Teichm\"uller modular group is countable. In this paper, we show that the Teichm\"uller modular group is countable or uncountable for some analytically infinite Riemann surfaces defined by generalized Cantor sets.
 \end{abstract}

\maketitle


\section{Introduction}

\subsection{Terminology of Riemann surfaces} 

We call a Riemann surface $X$ \textit{hyperbolic} if $X$ is represented by a quotient space $\mathbb{D}/\Gamma$ of the unit disk $\mathbb{D}$ by a torsion-free Fuchsian group $\Gamma$. In this paper, any Riemann surface is supposed to be hyperbolic. A Riemann surface $X$ is of \textit{analytically finite type} if $X$ is obtained from a compact surface by removing at most finitely many points, and $X$ is of \textit{analytically infinite type} if $X$ is not of analytically finite type. On the other hand, a Riemann surface $X$ is of \textit{topologically finite type} if the fundamental group $\pi_1 (X) \cong \Gamma$ is finitely generated, and $X$ is of \textit{topologically infinite type} if $X$ is not of topologically finite type. Also, a Fuchsian group $\Gamma$ is of \textit{the first kind} if the limit set of $\Gamma$ coincides with the unit circle: $\Lambda (\Gamma) = \partial \mathbb{D}$, and $\Gamma$ is of \textit{the second kind} if $\Lambda (\Gamma) \subsetneq \partial \mathbb{D}$. Now, if $X$ has the boundary, then we write it as $\partial X$. Also, a Fuchsian group $\Gamma$ acts properly discontinuously on $\overline{\mathbb{D}} \setminus \Lambda (\Gamma)$, so if $\Gamma$ is of the second kind, then we obtain a bordered Riemann surface $(\overline{\mathbb{D}} \setminus \Lambda(\Gamma))/\Gamma$ containing $X$ as its interior. We refer to  $(\partial \mathbb{D} \setminus \Lambda(\Gamma))/\Gamma$ as the \textit{boundary at infinity} of $X$ and write it as $\partial X$, too. 

\subsection{Teichm\"uller space and its Teichm\"uller modular group}

For a Riemann surface $X$, the \textit{Teichm\"uller space} $T(X)$ is the set of Teichm\"uller equivalence classes of quasiconformal mappings $f$ of $X$ onto another Riemann surface, where two quasiconformal mappings $f_1$ and $f_2$ are Teichm\"uller equivalent if there exists a conformal mapping $h: f_1 (X) \to f_2(X)$ such that $f_2^{-1} \circ h \circ f_1: X \to X$ is homotopic to the identity. If $\partial X \neq \emptyset$, the homotopy is considered to be relative to $\partial X$ ($:=rel. $ $\partial X$), that is, the homotopy fixes points of $\partial X$.  We write the Teichm\"uller equivalence class of $f$ as $[f]$. It is known that $T(X)$ has a complex Banach manifold structure, and if $X$ is of analytically finite type, then $\dim T(X) < \infty$; othewise $\dim T(X) = \infty$. On $T(X)$, a  distance between two points $[f_1]$ and $ [f_2]$ is defined by $d_T ([f_1], [f_2]) = \inf_f \log K(f)$, where the infimum is taken over all quasiconformal mappings from $f_1(X)$ to $f_2(X)$ homotopic to $f_2 \circ f_1^{-1}$ ($rel.$ $\partial X$ if $\partial X \neq \emptyset$), and $K(f)$ is the maximal dilatation of $f$. This is a complete distance on $T(X)$ and is called \textit{Teichm\"uller distance}. 

For a Riemann surface $X$, the \textit{quasiconformal mapping class group} MCG$(X)$ is defined as the group of all homotopy classes $[g]$ of quasiconformal automorphisms $g$ of $X$ ($rel.$ $\partial X$ if $\partial X \neq \emptyset$). For each $[g] \in$ MCG$(X)$, define the transformation $[g]_{*}$ of $T(X)$ as $[f] \mapsto [f \circ g^{-1}]$, then MCG$(X)$ acts on $T(X)$ isometrically with respect to $d_T$. Now, let Aut$(T(X))$ be the group of all isometric biholomorphic automorphisms of $T(X)$. We consider the homomorphism $\iota:$ MCG$(X) \to$ Aut$(T(X))$ defined by $[g] \mapsto [g]_{*}$ and define the \textit{Teichm\"uller modular group} for $X$, which is denoted by Mod$(X)$, as the image Im $\iota \subset$ Aut$(T(X))$ of $\iota$. Except for a few low-dimensional Teichm\"uller spaces, the homomorphism $\iota$ is injective (cf.  \cite{Epstein}, \cite{Matsuzaki2}) and surjective (cf. \cite{Markovic}). Therefore, in this paper, we identify the quasiconformal mapping class group with the Teichm\"uller modular group. 

In section 3, we think a bit about the \textit{reduced Teichm\"uller modular group}  Mod$^{\sharp} (X)$ for a Riemann surface $X$. This is the quotient group of Mod$(X)$ by free homotopy equivalence, that is, the homotopy does not necessarily fix points of $\partial X$ if $\partial X \neq \emptyset$. 

\subsection{Some Riemann surfaces of topologically infinite type and Teichm\"uller modular groups for them}

In 2003, Shiga (\cite{Shiga1}) considered two distances on the Teichm\"uller space $T(X)$; the Teichm\"uller distance $d_T$ and the length spectrum distance $d_L$. By the definition, the Teichm\"uller distance $d_T([f_1],[f_2])$ means how different the complex structures of two Riemann surfaces $f_1(X)$ and $f_2(X)$ are. On the other hand, though we do not describe the definition in this paper, the length spectrum distance $d_L([f_1],[f_2])$ means how different the hyperbolic structures of two Riemann surfaces $f_1(X)$ and $f_2(X)$ are. If $X$ is analytically finite Riemann surface, then the two distances $d_T$ and $d_L$ define the same topology on $T(X)$, but otherwise it is not always true. Shiga constructed a topologically infinite Riemann surface $S$ such that the two distances define different topologies on $T(S)$. His Riemann surface is essentially the same as the Riemann surface $S$ constructed as follows: let $\{ a_n \}_{n=1}^{\infty}$ be a monotonic divergent sequence of positive numbers such that $a_{n+1}> n a_n$, and let $\{ P_n \}_{n=1}^{\infty}$ be a sequence of pairs of pants such that the hyperbolic lengths of three boundary geodesics of $P_n$ are $a_n, a_{n+1}, a_{n+1}$ $(n=1,2,...)$. Firstly, make $2$ copies of $P_1$ and glue them together along the boundaries of length $a_1$, then we obtain a Riemann surface $S_1$ of type $(0,4)$. Secondly, make $4$ copies of $P_2$ and glue them to $S_1$ along the boundaries of length $a_2$, then we obtain a Riemann surface $S_2$ of type $(0,8)$. Inductively, for each $n$, make $2^n$ copies of $P_n$ and glue them to $S_{n-1}$ along the boundaries of length $a_n$, then we obtain a Riemann surface $S_n$ of type $(0,2^{n+1})$. We define the Riemann surface $S$ as the exhaustion of $\{S_n\}_{n=1}^{\infty}$, i.e.,  $S=\bigcup_{n=1}^{\infty} S_n$.  (By the way, he also showed that if a topologically infinite Riemann surface $X$ satisfies some condition, the two distances define the same topology on $T(X)$ in the same paper \cite{Shiga1}. And in 2018, we generalized his theorem, more precisely, we showed  that if $X$ is a Riemann surface with bounded geometry, then the two distances define the same topology on $T(X)$ (\cite{Kinjo3}).)   

In 2005, Matsuzaki (\cite{Matsuzaki1}) considered Shiga's Riemann surface $S$, a reconstructed Riemann surface $R$ from $S$ and the Teichm\"uller modular group Mod$(R)$ for $R$. Before mentioning it, we introduce a proposition for countability of the Teichm\"uller modular group.

\begin{prop}[(Proposition 1 of \cite{Matsuzaki1}]\label{prop1}
Suppose $X$ is a hyperbolic Riemann surface. If \rm{Mod}$(X)$ is countable, then $X=\mathbb{D}/\Gamma$ satisfies the following conditions.
\begin{enumerate}
\item The number of simple closed geodesics on $X$ of which lengths are smaller than $M$ for arbitrary $M >0$ is finite.
 
\item The Fuchsian group $\Gamma$ is of the first kind.
\end{enumerate}
\end{prop}

In \S 3 of \cite{Matsuzaki1}, Matsuzaki showed that if a Riemann surface $S$ is constructed by gluing above-mentioned pants $\{ P_n\}_{n=1}^{\infty}$ in the \textit{usual} way, then $S$ is not geodesically complete, that is, there exists a geodesic connecting $\partial P_1$ and $\partial P_n$ such that its length converges as $n \to \infty$. This means that the geodesic completion of $S$ does not coincide with $S$, hence the Fuchsian group corresponding to $S$ is of the second kind. (cf. Proposition 3.7 of \cite{Basmajian-saric}.) In particular, Mod$(S)$ is uncountable by Proposition \ref{prop1} (2). However,  if a Riemann surface $R$ is constructed by gluing above-mentioned pants $\{ P_n\}_{n=1}^{\infty}$ in a \textit{special} way, then $R$ is geodesically complete, so the geodesic completion of $R$ coincides with $R$. Here, a special way is to give each boundary geodesic of each pair of pants some amount of twist when we glue pants together. Then, the corresponding Fuchsian group is of the first kind, and also he could show that Mod$(R)$ is countable. 

\subsection{Generalized Cantor sets} 

Let $\{q_n\}_{n=1}^{\infty}$ be a sequence of numbers in $(0,1)$. Put $I:= [0,1] \subset \mathbb{R}$. A generalized Cantor set $E(\omega)$ for $\omega=\{q_n\}_{n=1}^{\infty}$ is defined as follows: Firstly, remove an open interval with the length $q_1$ from $I$ so that the remaining intervals $I_1^1, I_1^2 \subset I$ have the same length. Secondly, remove an open interval with the length $q_2 |I_1^1 |$ from each $I_1^i$ $(i=1,2)$ so that the remaining intervals $I_2^1,I_2^2, I_2^3, I_2^4 \subset I$ have the same length, where $|\cdot|$ means the length of the interval. Inductively, remove an open interval with the length $q_n |I_{n-1}^1|$ from each $I_{n-1}^i$ $(i=1,...,2^{n-1} )$ so that the remaining intervals $I_n^1,..., I_n^{2^n} \subset I$ have the same length. For each $n \in \mathbb{N}$, put $E_n = \bigcup_{i=1}^{2^n} I_n^{i}$. We define a generalized Cantor set $E(\omega)$ for $\omega$ as $\bigcap_{n=1}^{\infty} E_n$. In our previous paper (\cite{Kinjo4}), we considered the Riemann surface $X_{E(\omega)} := \hat{\mathbb{C}} \setminus E(\omega)$ (obtained from the Riemann sphere $ \hat{\mathbb{C}}$ by removing $E(\omega)$) and the Teichm\"uller space $T(X_{E(\omega)})$ of $X_{E(\omega)}$, and proved a theorem about the Teichm\"uller distance $d_T$ and the length spectrum distance $d_l$ on $T(X_{E(\omega)})$. In this paper, we consider the Teichm\"uller modular group for $X_{E(\omega)}$.

\subsection{Our results}

 At first, we give a sufficient condition for Mod$(X_{E(\omega)})$ to be uncountable. It is obtained by Proposition \ref{prop1} above and some lemma of our previous paper \cite{Kinjo4}.

\begin{theorem}\label{thm0} If there exists a subsequence $\{ q_{n(k)} \}_{k=1}^{\infty}$ of $\omega= \{q_n\}_{n=1}^{\infty}$ such that $q_{n(k)} >c$ for some constant $c \in (0,1)$, then the Teichm\"uller modular group for the Riemann surface $X_{E(\omega)}$ is uncountable. In particular, if $\inf_n q_n \neq 0$, then \rm{Mod}$(X_{E(\omega)})$ is uncountable.     
\end{theorem} 

Not only $\omega$ such that $\inf_n q_n \neq 0$ but also some $\omega$ such that $\inf_n q_n =0$ satisfies the condition of Theorem \ref{thm0}. For example, let $\omega=\{ q_n \}_{n=1}^{\infty}$ be a sequence defined by 
\[
 q_n =
 \begin{cases}
 \frac{1}{2} & (n=2m-1; m \in \mathbb{N})\\
 (\frac{1}{2})^n & (n=2m; m \in \mathbb{N}).
 \end{cases}
 \]
Then $\inf_n q_n =0$ and there exists a subsequence $\{ q_{2m} \}_{m=1}^{\infty}$ of $\omega$ such that $q_{n(k)} >1/3$.

Next, we give a sufficient condition for Mod$(X_{E(\omega)})$ to be countable. In Theorem 1.1 of our previous paper (\cite{Kinjo4}), we considered two conditions (I),(II) for $\omega$ such that $\inf_n q_n =0$, and showed that if $\omega$ satisfies either (I) or (II), then the two distances $d_T$ and $d_L$ define the different topologies on $T(X_{E(\omega)})$. Now, if $\omega$ satisfies (II), then it satisfies the condition of Theorem \ref{thm0} above, too. On the other hand, if $\omega$ satisfies (I), then it does not do so.  In this paper, our main theorem below says that if $\omega$ satisfies (I), then \rm{Mod}$(X_{E(\omega)})$ is countable:

\begin{theorem}\label{thm1}
If the sequence $\omega$ satisfying $q_n \cdot \log(\log (1/q_{n+1})) \to \infty$ as $n \to \infty$, then \rm{Mod}$(X_{E(\omega)})$ is countable.
\end{theorem}
The sequence $\omega$ satisfying Theorem \ref{thm1} converges to $0$ very rapidly. The following is an example of such sequences which is a little different from Example 1.2 of \cite{Kinjo4}. 
\begin{example}
Take a sequence $\omega=\{ q_n \}_{n=1}^{\infty}$ so that $q_{n+1} = 1/\exp (n^{1/q_n})$. Then
\[ q_n \cdot \log(\log (1/q_{n+1})) = q_n \cdot (1/q_n) \log n = \log n \to \infty \]
as $n \to \infty$.
\end{example}

The merit of $X_{E(\omega)}$ is the  following:

\begin{prop}\label{prop2}
For any $\omega$, the Fuchsian group $\Gamma$ corresponding to $X_{E(\omega)}$ is of the first kind.
\end{prop}

By the property, we can construct the analytically infinite Riemann surface for which the Teichm\"uller modular group is countable without caring about twist of boundary geodesics of pairs of pants. In section 2, we prove Theorem \ref{thm0} and Proposition \ref{prop2}. In section 3, we prove Theorem \ref{thm1}.

\textbf{Acknowledgement.} The author would like to thank Professor H. Shiga for his valuable comments at the seminar.


\section{Proofs of Theorem \ref{thm0} and Proposition \ref{prop2}}

We decompose $X_{E(\omega)}$ into pairs of pants as we (\S 2 of \cite{Kinjo4}) or Shiga (\S 3 of \cite{Shiga2}) did. Recall that for a sequence $\omega=\{ q_n \}_{n=1}^{\infty}$, the generalized Cantor set $E(\omega)$ is $\bigcap_{n=1}^{\infty} E_n$, where $E_n$ is the union of closed intervals $\{ I_n^{i} \}_{i=1}^{2^n}$ in $I=[0,1]$ ($n=1,2,...$). Now, for each $n \in \mathbb{N}$ and each $i \in \{1,...,2^n \}$, let $\gamma_n^i$ be a simple closed curve separating $I_n^i$ from other intervals $I_n^{i'}$, $i' \in \{1,...,2^n \} \setminus \{i \}$. Then, in $E(\omega)$, let $[\gamma_n^i]$ be the simple closed geodesic which is freely homotopic to $\gamma_n^i$, where $[\gamma_1^1] = [\gamma_1^2]$, so we write $[\gamma_1]$ for the geodesic. Let $P_1^i$ be a pair of pants with boundary geodesics $[\gamma_1]$, $[\gamma_2^{2i-1}]$ and $[\gamma_2^{2i}]$ $(i=1,2)$. And for each $n \in \mathbb{N}$ and $i \in \{ 1,...,2^n \}$, let $P_n^i$ be a pair of pants with boundary geodesics $[\gamma_n^i]$, $[\gamma_{n+1}^{2i-1}]$ and $[\gamma_{n+1}^{2i}]$. (See Figure \ref{pd2024}.) Then we obtain a pants decomposition: $X_{E(\omega)}= \bigcup_{n=1}^{\infty} ( \bigcup_{i=1} ^{2^n} P_n^i )$. We call this a natural pants decomposition of $X_{E(\omega)}$.

\begin{figure}[h]
\centering
\includegraphics[width=11 cm,bb=0 0 825 312]{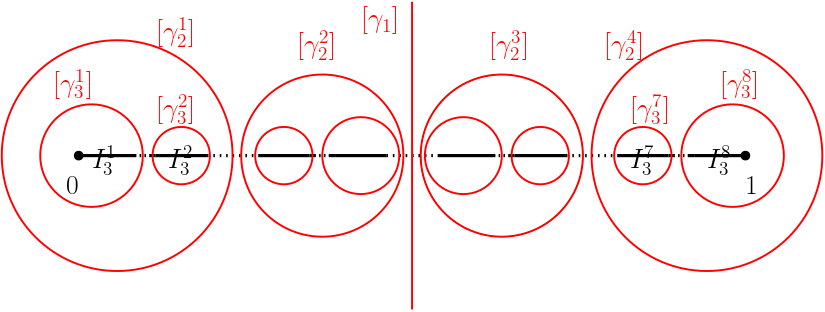}\\
\caption{Pairs of pants $\bigcup_{n=1}^3 ( \bigcup_{i=1} ^{2^n} P_n^i )$}
\label{pd2024}
\end{figure}   

To prove Theorem \ref{thm0}, we use the lemma below, which is a part of some lemma in our previous paper \cite{Kinjo4}. Here, $\ell_{X} (\gamma)$ means the hyperbolic length of a curve $\gamma$ on a hyperbolic Riemann surface $X$. 

\begin{lemma}[Lemma 2.1 (1) of \cite{Kinjo4}]\label{lemma1}  For any $\omega$ and any $n \ge 1$, 
\[ \ell_{X_{E(\omega)}} ([\gamma_n^1]) < \frac{\pi^2}{\tanh^{-1} q_n} \]
holds.
\end{lemma}   

\begin{remark}
For each $n$, there are $2^n$ simple closed geodesics $\{ [\gamma_n^i]\}_{i=1}^{2^n}$ as the $n$-th geodesic, and in the case where $2 \le i \le 2^{n-1}$, $\ell_{X_{E(\omega)}} ([\gamma_n^i]) < \pi^2/\tanh^{-1} q_n$ does not hold, in general. (cf. Lemma 2.1 (2) of \cite{Kinjo4}.) However, if $\omega$ is monotonic decreasing, then $\ell_{X_{E(\omega)}} ([\gamma_n^i]) < \pi^2/\tanh^{-1} q_n$ holds for any $n \in \mathbb{N}$ and $i \in \{1,...,2^n \}$.  (cf. Remark 3 of \cite{Kinjo4}.) 
\end{remark}

From the above, Theorem \ref{thm0} is immediately proved. 

\begin{proof}[Proof of Theorem \ref{thm0}]
By Lemma \ref{lemma1}, if there exists a subsequence $\{ q_{n(k)} \}_{k=1}^{\infty}$ of $\omega$ such that $q_{n(k)} >c$, then $\ell_{X_{E(\omega)}} ([\gamma_{n(k)}^1]) < \pi^2/\tanh^{-1} c$ for $k =1,2,...$ Hence, by Proposition \ref{prop1} (1), Mod$(X_{E(\omega)})$ is uncountable.
\end{proof}

Before proving Theorem \ref{thm1}, we check Proposition \ref{prop2}. The following proof is based on the idea of Professor H. Shiga (because the original proof the author wrote turned out to be wrong at the seminar).

\begin{proof}[Proof of Proposition \ref{prop2}]
Assume that the Fuchsian group $\Gamma$ is of the second kind for some $\omega$. Then, for any point $p \in \mathbb{D}$, there exists a geodesic ray $\hat{r}_p$ in $\mathbb{D}$ starting at $p$ such that $\ell_{\mathbb{D}} (\hat{r}_p)$ is infinite, but $\ell_{X_{E(\omega)}} (\pi (\hat{r}_p) \cap X_{E(\omega)})$ is finite, where $\pi : \mathbb{D} \to \mathbb{D}/\Gamma$ is the the universal covering. Put $r_p := \pi (\hat{r}_p)$. Let $\{ P_n^{*} \}_{n=k_p}^{\infty}$ be a family of the pants (of the natural pants decomposition) containing $r_p$ and $\{ \gamma_n\}_{n=k_p}^{\infty}$ be a component of each $\partial P_n^{*}$ intersecting $r_p$. Here, $E(\omega)$ is totally disconnected, hence the diameter of $\gamma_n$ in $\mathbb{C}$ converges to $0$ as $n \to \infty$,  and there exists a point $p_{\infty} \in E(\omega)$ such that $\gamma_n \to p_{\infty}$. Now take points $p_1 , p_2 \in E(\omega)$ which are contained in $\hat{\mathbb{C}} \setminus \bigcup_{n= k_p}^{\infty} P_n^{*}$ and put $W:=\hat{\mathbb{C}} \setminus \{p_{\infty}, p_1 , p_2 \}$. Since $r_p$ goes to $p_{\infty}$, we have $\ell_{W} (r_p) = \infty$. However, $X_{E(\omega)} \subset W$, so $\ell_{X_{E(\omega)}} (r_p) \ge \ell_{W} (r_p) $ and this is a contradiction.  
\end{proof}


\section{Proof of Theorem \ref{thm1}} 

Let $\omega =\{ q_n \}_n^{\infty}$ be a sequence of numbers in $(0,1)$, and let $[ \gamma_n^i]$ and $ P_n^i$ ($n \in \mathbb{N}$, $i \in \{ 1,...,2^n \}$) be a closed geodesic and a pair of pants of $X_{E(\omega)}$ taken in Section 2, respectively. For each $n \in \mathbb{N}$, we take the subsurface $X_n := \bigcup_{k=1}^n (\bigcup_{i=1}^{2^k} P_k^i)$ of $X_{E(\omega)}$. Firstly, we show the following lemma.

\begin{lemma}\label{lemma3.1}
Let $\omega$ be the sequence satisfying the condition of Theorem \ref{thm1}. Then, for any $K$-quasiconformal automorphism $g: X_{E(\omega)} \to X_{E(\omega)}$, there exists $n_1 \in \mathbb{N}$ such that if $n \ge n_1$, then the image $g(X_n) $ of $X_n$ is freely homotopic to $X_n$ in $X_{E(\omega)}$, that is, each component of $\partial g(X_n)$ is homotopic to some component of $\partial X_n$.
\end{lemma}  

To prove this, we use Lemma \ref{lemma2.3}, (\cite{Kinjo4}), Lemma \ref{lemma2.4} (\cite{Kinjo4}) and Lemma \ref{wol} (Wolpert's Lemma).  

\begin{lemma}[Lemma 2.3 of \cite{Kinjo4}]\label{lemma2.3}
Let $\omega$ be the sequence satisfying the condition of Theorem \ref{thm1}. Then,  for any $i \in \{ 1,...,2^n \}$ and $j \in \{ 1,...,2^{n+1} \}$,
\[ \frac{\ell_{X_{E(\omega)}} ([\gamma_{n+1}^j])}{\ell_{X_{E(\omega)}} ([\gamma_{n}^i])} \to \infty\]
as $n \to \infty$.
\end{lemma}

Below, $s_n^i (\subset \mathbb{R} \cup \{ \infty\})$ is the shortest geodesic segment connecting $[\gamma_{n+1}^{2i-1}]$ and $[\gamma_{n+1}^{2i}]$ in each pair of pants $P_n^i$ with boundary geodesics $\{ [\gamma_n^i], [\gamma_{n+1}^{2i-1}], [\gamma_{n+1}^{2i}]\}$, and $d (\cdot ,\cdot )$ is the hyperbolic distance on $X_{E(\omega)}$.

\begin{lemma}[Lemma 2.4 of \cite{Kinjo4}]\label{lemma2.4}
Let $\omega$ be the sequence satisfying the condition of Theorem \ref{thm1}. Then,  for any $i \in \{ 1,...,2^n \}$,
\[ \frac{d([\gamma_n^i], s_n^i)}{\ell_{X_{E(\omega)}} ([\gamma_{n}^i])} \to \infty\]
as $n \to \infty$.
\end{lemma}

\begin{lemma}[\cite{Wolpert}]\label{wol}
Let $f: X \to X'$ be a $K$-quasiconformal homeomorphism from a hyperbolic Riemann surface $X$ onto another hyperbolic Riemann surface $X'$. And let $\gamma$ be a closed geodesic on $X$ and $[f(\gamma)]$ be the geodesic of the free homotopy class of $f(\gamma)$. Then
\[ \frac{1}{K} \le \frac{\ell_{X'} ([f(\gamma)])}{\ell_{X} ([\gamma])} \le K\]
\end{lemma}

\begin{proof}[Proof of Lemma \ref{lemma3.1}] 
Note that for $K \ge 1$, there exists $n_1 \in \mathbb{N}$ such that if $n \ge n_1$, then
\[ \frac{d([\gamma_n^i], s_n^i)}{\ell_{X_{E(\omega)}} ([\gamma_{n}^i])} > K\]
for any $i \in \{1,...,2^n \}$ by Lemma \ref{lemma2.4}. Now, assume that for any $N \in \mathbb{N}$, $g(X_n)$ is not freely homotopic to $X_n$ in $X_{E(\omega)}$ for some $n \ge N$. Then, for some $n \ge n_1$, $g(X_n)$ is not freely homotopic to $X_n$, so there exists a component $\gamma_n$ of $\partial X_n$ such that $[g(\gamma_n)]$ crosses $\partial X_n$, where $[g(\gamma_n)]$ is the closed geodesic freely homotopic to $g(\gamma_n)$. Then, the length of $[g(\gamma_n)]$ is larger than $d([\gamma_{n+m}^j], s_{n+m}^j)$ for some $m \ge 0$ and $j \in \{1,..., 2^{n+m} \}$ since $[g([\gamma_n])$ crosses $[\gamma_{n+m}^j]$ and $s_{n+m}^j$. Therefore, $\ell_{X_{E(\omega)}} ([g(\gamma_n)]) > d([\gamma_{n+m}^j], s_{n+m}^j) >K \ell_{X_{E(\omega)}} ([\gamma_{n+m}^j]) \ge K \ell_{X_{E(\omega)}} ([\gamma_{n}^i])$ holds by Lemmas \ref{lemma2.4} and \ref{lemma2.3}. It contradicts Lemma \ref{wol}.  
\end{proof}

Next, we consider a Dehn twist about each component of $\partial X_{n}$ for a sufficiently large number $n$.

\begin{lemma}\label{lemma3.5}
Let $\omega$ be the sequence satisfying the condition of Theorem \ref{thm1}. For for any $K$-quasiconformal automorphism $g: X_{E(\omega)} \to X_{E(\omega)}$, there exists $n_2 \in \mathbb{N}$ such that  if $n \ge n_2$, on any component of $\partial X_n$, $g$ does not cause nor a half Dehn twist, a Dehn twist nor multiple twists. 
\end{lemma}

We use a lemma of our previous paper \cite{Kinjo4} and a theorem of Matsuzaki (\cite{Matsuzaki2003}). In the following, $\eta$ is the collar function: $\eta(x)=\sinh^{-1} (1/\sinh (x/2))$.

\begin{lemma}[Lemma 2.2 of \cite{Kinjo4}]\label{lemma2.2}
Let $\omega =\{ q_n \}_{n=1}^{\infty}$ be an arbitrary sequence of numbers in $(0,1)$. For any $n \in \mathbb{N}$ and $i \in \{1,...,2^n \}$,
\[ \ell_{X_{E(\omega)}} ([\gamma_n^i]) > 2\eta  \left(\frac{\pi^2}{ \log ( (1+q_n)/(2q_n))}\right)\]
holds.
\end{lemma}

\begin{theorem}[Part of Theorem 1 of \cite{Matsuzaki2003}]\label{thm3.7}
Let $\gamma$ be a simple closed geodesic on a Riemann surface $X$ and $f$ be $n$-times Dehn twist about $\gamma$. Then the maximal dilatation of an extremal quasiconformal automorphism of $f$ satisfies 
\[K(f) \ge \sqrt{\{(2|n|-1)\ell_{X}(\gamma)/\pi\}^2 +1}.\]
\end{theorem}

\begin{proof}[Proof of Lemma \ref{lemma3.5}]
Since $\omega$ satisfies the condition of Theorem \ref{thm1}, $q_n \to 0$ as $n \to \infty$, therefore $\ell_{X_{E(\omega)}} ([\gamma_n^i]) \to \infty$ for any $i$ by Lemma \ref{lemma2.2}. Hence, for $K \ge 1$, there exists $n_2 \in \mathbb{N}$ such that  $ \ell_{X_{E(\omega)}} ([\gamma_n^i]) > \pi K^2$ if $n \ge n_2$. Assume that $g$ cause a half Dehn twist $f_n$ on some component $\gamma_n$ of $\partial X_n$ for some $n \ge n_2$. Then $f_n^2 := f_n \circ f_n$ is a Dehn twist about $\gamma_n$, so the maximal dilatation $K(f_n^2)$ of $f_n^2$ is larger than $\sqrt{K^4 +1} > K^2$ by Theorem \ref{thm3.7}. Since $K(f_n^2) \le K(f_n)^2$, we have $K <K(f_n)$, and this is a contradiction. From above, $g$ does not cause nor a Dehn twist nor multiple twists.
\end{proof}

Finally we prove the main theorem.

\begin{proof}[Proof of Theorem \ref{thm1}]
For $K$-quasiconformal automorphism $g: X_{E(\omega)} \to X_{E(\omega)}$, put $N:= \max\{n_1, n_2 \}$, where $n_1, n_2$ are numbers of Lemma \ref{lemma3.1} and Lemma \ref{lemma3.5}, respectively. Then, for any $n \ge N$, $g(X_{n} \setminus X_{n-1})$ is homotopic to $X_{n} \setminus X_{n-1} =\bigsqcup_{i=1}^{2^n} P_n^i$ in $X_{E(\omega)}$, and  on any component of $\partial X_n$, $g$ does not cause nor a half Dehn twist, a Dehn twist nor multiple twists.

Now, for an arbitrary $K \in \mathbb{N}$, let Mod$(X_{E(\omega)})_K$ be a subset of the Teichm\"uller modular group Mod$(X_{E(\omega)})$ such that each element has $K$-quasiconformal automorphism $g$ as a representative. From above, Mod$(X_{E(\omega)})_K$ is embedded in the reduced Teichm\"uller modular group $\text{Mod}^{\sharp}(X_N)$  for the bordered Riemann surface $X_N$. Indeed, $g$ is determined by $g|_{X_N}$, that is, for quasiconformal automorphisms $g_1, g_2$ of $X_{E(\omega)}$, if $g_1 |_{X_N} = g_2 |_{X_N}$, then $[g_1] =[g_2]$.

Since $X_N$ is topologically finite, $\text{Mod}^{\sharp} (X_N)$ is finitely generated, thus countable. Hence Mod$(X_{E(\omega)})_K$ is countable for any $K \in \mathbb{N}$, and Mod$(X_{E(\omega)})$ is countable, too.    
\end{proof}

\end{document}